\documentclass[english,oneside]{amsart}
\usepackage[T1]{fontenc}
\usepackage[latin1]{inputenc}
\usepackage{amsthm}
\usepackage{amssymb}

\makeatletter
\numberwithin{equation}{section} 
\numberwithin{figure}{section} 
\theoremstyle{plain}
\newtheorem{thm}{Theorem}
  \theoremstyle{definition}
  \newtheorem{defn}[thm]{Definition}
  \theoremstyle{plain}
  \newtheorem{cor}[thm]{Corollary}
  \theoremstyle{remark}
  \newtheorem{rem}[thm]{Remark}
  \theoremstyle{plain}
  \newtheorem{lem}[thm]{Lemma}

\makeatother

\usepackage{babel}

\begin{document}

\title{Scaled Asymptotics For Some $q$-Series As $q$ Approaches One }

\author{Ruiming Zhang}


\email{ruimingzhang@yahoo.com}
\begin{abstract}
In this work we investigate Plancherel-Rotach type asymptotics for
some $q$-series as $q\to1$. These $q$-series generalize Ramanujan
function $A_{q}(z)$ ($q$-Airy function), Jackson's $q$-Bessel function
$J_{\nu}^{(2)}$(z;q), Ismail-Masson orthogonal polynomials($q^{-1}$-Hermite
polynomials) $h_{n}(x|q)$, Stieltjes-Wigert orthogonal polynomials
$S_{n}(x;q)$, $q$-Laguerre orthogonal polynomials $L_{n}^{(\alpha)}(x;q)$
and confluent basic hypergeometric series. 
\end{abstract}

\curraddr{School of Mathematics\\
Guangxi Normal University\\
Guilin City, Guangxi 541004\\
P. R. China.}


\maketitle

\section{Introduction\label{sec:Introduction}}

In \cite{Zhang1} we derived certain Plancherel-Rotach type asymptotics
for some $q$-series. These $q$-series generalize Ramanujan's entire
function $A_{q}(z)$, Jackson's $q$-Bessel function $J_{\nu}^{(2)}$(z;q),
Ismail-Masson orthogonal polynomials ($q^{-1}$-Hermite polynomials)
$h_{n}(x|q)$, Stieltjes-Wigert orthogonal polynomials $S_{n}(x;q)$,
$q$-Laguerre orthogonal polynomials $L_{n}^{(\alpha)}(x;q)$ and
confluent basic hypergeometric series. 

In this work we employ the method used in \cite{Zhang2} to study
the scaled asymptotics of these $q$-series as $q\to1$. In section
\S\ref{sec:Preliminaries} we list some notations. We present our
results in section \S\ref{sec:Main-Results} and their proofs in
section \S\ref{sec:Proofs}. Throughout this work we always assume
that $0<q<1$ unless otherwise stated.

\section{Preliminaries\label{sec:Preliminaries}}

For a complex number $z$, we define \cite{Andrews,Gasper,Ismail2,Koekoek}\begin{align*}
(z;q)_{\infty} & =\prod_{k=0}^{\infty}(1-zq^{k}),\end{align*}
and the $q$-Gamma function is defined as\[
\Gamma_{q}(z)=\frac{(q;q)_{\infty}}{(q^{z};q)_{\infty}}(1-q)^{1-z}\quad z\in\mathbb{C}.\]
The $q$-shifted factorials of $a,a_{1},\dotsc a_{m}$ are given by

\begin{align*}
(a;q)_{n} & =\frac{(a;q)_{\infty}}{(aq^{n};q)_{\infty}},\quad(a_{1},\dotsc,a_{m};q)_{n}=\prod_{k=1}^{m}(a_{k};q)_{n}\end{align*}
for all integers $n\in\mathbb{Z}$ and $m\in\mathbb{N}$. Given two
sets of complex numbers $\left\{ a_{1},\dots,a_{r}\right\} $ and
$\left\{ b_{1},\dotsc,b_{s}\right\} $, the basic hypergeometric series
${}_{r}\phi_{s}$ is formally defined as\begin{align*}
{}_{r}\phi_{s}\left(\begin{array}{c}
a_{1},\dotsc,a_{r}\\
b_{1},\dotsc,b_{s}\end{array}\vert q,z\right) & =\sum_{k=0}^{\infty}\frac{(a_{1},\dotsc,a_{r};q)_{k}(zq^{-\ell})^{k}q^{\ell k^{2}}}{(q,b_{1},\dotsc,b_{s};q)_{k}(-1)^{k(s+1-r)}},\end{align*}
where \[
\ell=\frac{s+1-r}{2},\]
and it is a confluent basic hypergeometric series if $\ell>0$. 

Given nonnegative integers $r,s,t$ and a positive number $\ell$,
we define \cite{Zhang1} \begin{align*}
 & g(a_{1},\dotsc,a_{r};b_{1},\dotsc,b_{s};q;\ell;z)\\
 & =\sum_{k=0}^{\infty}\frac{(q^{k+1},b_{1}q^{k},\dotsc,b_{s}q^{k};q)_{\infty}q^{\ell k^{2}}(-z)^{k}}{(a_{1}q^{k},\dotsc,a_{r}q^{k};q)_{\infty}},\\
 & h(a_{1},\dotsc,a_{r};b_{1},\dotsc,b_{s};c_{1},\dots,c_{t};q;\ell;z)\\
 & =\sum_{k=0}^{n}\frac{(q^{k+1},b_{1}q^{k},\dotsc,b_{s}q^{k};q)_{\infty}q^{\ell k^{2}}(-z)^{k}}{(a_{1}q^{k},\dotsc,a_{r}q^{k};q)_{\infty}}\frac{(q,c_{1},\dotsc,c_{t};q)_{n}}{(q,c_{1},\dotsc,c_{t};q)_{n-k}},\end{align*}
 where \begin{align}
0 & \le a_{1},\dotsc,a_{r},b_{1},\dotsc,b_{s},c_{1},\dotsc,c_{t}<1.\label{eq:g-h condition}\end{align}
Jackson's $q$-Bessel function $J_{\nu}^{(2)}(z;q)$ is defined as
\cite{Ismail2,Gasper,Koekoek} \begin{align*}
J_{\nu}^{(2)}(z;q) & =\frac{(q^{\nu+1};q)_{\infty}}{(q;q)_{\infty}}\sum_{k=0}^{\infty}\frac{q^{k^{2}+k\nu}(-1)^{k}}{(q,q^{\nu+1};q)_{k}}\left(\frac{z}{2}\right)^{2k+\nu},\quad\nu>-1.\end{align*}
The Ismail-Masson polynomials $\left\{ h_{n}(x|q)\right\} _{n=0}^{\infty}$
are defined as \cite{Ismail2} \begin{align*}
h_{n}(\sinh\xi|q) & =\sum_{k=0}^{n}\frac{(q;q)_{n}q^{k(k-n)}(-1)^{k}e^{(n-2k)\xi}}{(q;q)_{k}(q;q)_{n-k}}.\end{align*}
Stieltjes\emph{-}Wigert orthogonal polynomials $\left\{ S_{n}(x;q)\right\} _{n=0}^{\infty}$
are defined as \cite{Ismail2}\begin{align*}
S_{n}(x;q) & =\sum_{k=0}^{n}\frac{q^{k^{2}}(-x)^{k}}{(q;q)_{k}(q;q)_{n-k}}.\end{align*}
The $q$-Laguerre orthogonal polynomials $\left\{ L_{n}^{(\alpha)}(x;q)\right\} _{n=0}^{\infty}$
are defined as \cite{Ismail2} \begin{align*}
L_{n}^{(\alpha)}(x;q) & =\sum_{k=0}^{n}\frac{q^{k^{2}+\alpha k}(-x)^{k}(q^{\alpha+1};q)_{n}}{(q;q)_{k}(q,q^{\alpha+1};q)_{n-k}}\end{align*}
for $\alpha>-1$. Clearly, we have\begin{align*}
A_{q}(z) & =\frac{g(-;-;q;1;z)}{(q;q)_{\infty}},\\
J_{\nu}^{(2)}(z;q) & =\frac{g(-;q^{\nu+1};q;1;z^{2}q^{\nu}/4)}{(q;q)_{\infty}^{2}(2/z)^{\nu}},\\
h_{n}(\sinh\xi\vert q) & =\frac{h(-;-;-;q;1;e^{-2\xi}q^{-n})}{e^{-n\xi}(q;q)_{\infty}},\\
S_{n}(x;q) & =\frac{h(-;-;-;q;1;x)}{(q;q)_{n}(q;q)_{\infty}},\\
L_{n}^{(\alpha)}(x;q) & =\frac{h(-;-;q^{\alpha+1};q;1;xq^{\alpha})}{(q;q)_{n}(q;q)_{\infty}},\\
{}_{r}\phi_{s}\left(\begin{array}{c}
a_{1},\dotsc,a_{r}\\
b_{1},\dotsc,b_{s}\end{array}\vert q,z(-1)^{s-r}\right) & =\frac{(a_{1},\dotsc,a_{r};q)_{\infty}g(a_{1},\dotsc,a_{r};b_{1},\dotsc b_{s};q;\ell;zq^{-\ell})}{(q,b_{1},\dotsc,b_{s};q)_{\infty}}.\end{align*}
The four Jacobi theta functions are \cite{Rademarcher} \begin{align*}
\theta_{1}(v|\tau) & =-i\sum_{k=-\infty}^{\infty}(-1)^{k}q^{(k+1/2)^{2}}e^{(2k+1)\pi iv},\\
\theta_{2}(v|\tau) & =\sum_{k=-\infty}^{\infty}q^{(k+1/2)^{2}}e^{(2k+1)\pi iv},\\
\theta_{3}(v|\tau) & =\sum_{k=-\infty}^{\infty}q^{k^{2}}e^{2k\pi iv},\\
\theta_{4}(v|\tau) & =\sum_{k=-\infty}^{\infty}(-1)^{k}q^{k^{2}}e^{2k\pi iv},\end{align*}
 where\begin{align*}
q & =e^{\pi i\tau},\quad\Im(\tau)>0.\end{align*}
 For our convenience, we also use the following notations\begin{align*}
\theta_{\lambda}(z;q) & =\theta_{\lambda}(v|\tau),\quad z=e^{2\pi iv},\quad q=e^{\pi i\tau}\end{align*}
 with \begin{align*}
\lambda & =1,2,3,4.\end{align*}
 By the Jacobi's triple product formula it follows that\begin{align*}
\theta_{1}(v|\tau) & =2q^{1/4}\sin\pi v(q^{2};q^{2})_{\infty}(q^{2}e^{2\pi iv};q^{2})_{\infty}(q^{2}e^{-2\pi iv};q^{2})_{\infty},\\
\theta_{2}(v|\tau) & =2q^{1/4}\cos\pi v(q^{2};q^{2})_{\infty}(-q^{2}e^{2\pi iv};q^{2})_{\infty}(-q^{2}e^{-2\pi iv};q^{2})_{\infty},\\
\theta_{3}(v|\tau) & =(q^{2};q^{2})_{\infty}(-qe^{2\pi iv};q^{2})_{\infty}(-qe^{-2\pi iv};q^{2})_{\infty},\\
\theta_{4}(v|\tau) & =(q^{2};q^{2})_{\infty}(qe^{2\pi iv};q^{2})_{\infty}(qe^{-2\pi iv};q^{2})_{\infty}.\end{align*}
 The Jacobi $\theta$ functions satisfy transformations\begin{align*}
\theta_{1}\left(\frac{v}{\tau}\mid-\frac{1}{\tau}\right) & =-i\sqrt{\frac{\tau}{i}}e^{\pi iv^{2}/\tau}\theta_{1}\left(v\mid\tau\right),\\
\theta_{2}\left(\frac{v}{\tau}\mid-\frac{1}{\tau}\right) & =\sqrt{\frac{\tau}{i}}e^{\pi iv^{2}/\tau}\theta_{4}\left(v\mid\tau\right),\\
\theta_{3}\left(\frac{v}{\tau}\mid-\frac{1}{\tau}\right) & =\sqrt{\frac{\tau}{i}}e^{\pi iv^{2}/\tau}\theta_{3}\left(v\mid\tau\right),\\
\theta_{4}\left(\frac{v}{\tau}\mid-\frac{1}{\tau}\right) & =\sqrt{\frac{\tau}{i}}e^{\pi iv^{2}/\tau}\theta_{2}\left(v\mid\tau\right).\end{align*}
The Euler Gamma function $\Gamma(z)$ is given by \cite{Andrews,Gasper,Ismail2,Koekoek}
\begin{align*}
\frac{1}{\Gamma(z)} & =z\prod_{k=1}^{\infty}\left(1+\frac{z}{k}\right)\left(1+\frac{1}{k}\right)^{-z},\quad z\in\mathbb{C}.\end{align*}
For any real number $x$, we have \begin{align*}
x & =\left\lfloor x\right\rfloor +\left\{ x\right\} ,\end{align*}
where the fractional part of $x$ is $\left\{ x\right\} \in[0,1)$
and $\left\lfloor x\right\rfloor \in\mathbb{Z}$ is the greatest integer
less or equal to $x$. The arithmetic function \begin{align*}
\chi(n) & =\begin{cases}
1 & 2\nmid n\\
0 & 2\mid n\end{cases},\end{align*}
 which is the principal character modulo $2$, satisfies the identities\begin{align*}
\chi(n) & =2\left\{ \frac{n}{2}\right\} =n-2\left\lfloor \frac{n}{2}\right\rfloor =\left\lfloor \frac{n+1}{2}\right\rfloor -\left\lfloor \frac{n}{2}\right\rfloor .\end{align*}
 Thus,\begin{align*}
\left\lfloor \frac{n+1}{2}\right\rfloor  & =\frac{n+\chi(n)}{2},\end{align*}
 and\begin{align*}
\left\lfloor \frac{n}{2}\right\rfloor  & =\frac{n-\chi(n)}{2}.\end{align*}

\section{Main Results\label{sec:Main-Results}}
\begin{defn}
\label{def:admissible-scale}An admissible scale is a sequence $\left\{ \lambda_{n}\right\} _{n=1}^{\infty}$
of positive numbers such that\begin{align}
\lim_{n\to\infty}\frac{\lambda_{n}}{\log n} & =\infty,\quad\lim_{n\to\infty}\frac{n}{\lambda_{n}^{2}}=\infty.\label{eq:admissible}\end{align}

\end{defn}
Clearly, \begin{align*}
\lambda_{n} & =n^{\beta}\log^{\gamma}n,\quad0<\beta<\frac{1}{2},\quad\gamma\ge0,\end{align*}
and\begin{align*}
\lambda_{n} & =\log^{\gamma}n,\quad\gamma>1\end{align*}
are admissible scales.

\subsection{$g$-function}

To simplify the type setting in the following theorem, we let 

\begin{align*}
g(z;q) & =g(a_{1},\dotsc,a_{r};b_{1},\dotsc,b_{s};q;\ell;z).\end{align*}

\begin{thm}
\label{thm:g-asymptotics-1}Given an admissible scale $\lambda_{n}$,
assume that \begin{align*}
z & =e^{2\pi v},\quad q=e^{-\pi\lambda_{n}^{-1}},\quad\ell>0,\quad v\in\mathbb{R},\end{align*}
 and\begin{align*}
a_{j} & =q^{\alpha_{j}},\quad b_{k}=q^{\beta_{k}},\quad\alpha_{j},\beta_{k}>0\end{align*}
for \begin{align*}
1 & \le j\le r,\quad1\le k\le s.\end{align*}
Then, \begin{align*}
g(-q^{-4n\ell}z;q) & =\exp\left\{ \frac{\pi\lambda_{n}}{\ell}\left(v+\frac{2n\ell}{\lambda_{n}}\right)^{2}\right\} \\
 & \times\sqrt{\frac{\lambda_{n}}{\ell}}\left\{ 1+\mathcal{O}(e^{-\ell^{-1}\pi\lambda_{n}})\right\} ,\end{align*}
 and\begin{align*}
g(q^{-4n\ell}z;q) & =\exp\left\{ \frac{\pi\lambda_{n}}{\ell}\left(v+\frac{2n\ell}{\lambda_{n}}\right)^{2}-\frac{\pi\lambda_{n}}{4\ell}\right\} \\
 & \times2\sqrt{\frac{\lambda_{n}}{\ell}}\left\{ \cos\frac{\pi\lambda_{n}v}{\ell}+\mathcal{O}\left(e^{-2\ell^{-1}\pi\lambda_{n}}\right)\right\} \end{align*}
 as $n\to\infty$, and the $\mathcal{O}$-term is uniform for $v$
in any compact subset of $\mathbb{R}$. 
\end{thm}
For the Ramanujan's entire function we have:
\begin{cor}
\label{cor:ramanujan}Given an admissible scale $\lambda_{n}$ , assume
that\begin{align*}
z & =e^{2\pi v},\quad q=e^{-\pi\lambda_{n}^{-1}},\quad v\in\mathbb{R},\end{align*}
we have \begin{align*}
A_{q}(-q^{-4n}z) & =\exp\left\{ \pi\lambda_{n}\left(v+\frac{2n}{\lambda_{n}}\right)^{2}+\frac{\pi\lambda_{n}}{6}-\frac{\pi}{24\lambda_{n}}\right\} \\
 & \times\frac{1}{\sqrt{2}}\left\{ 1+\mathcal{O}\left(e^{-\pi\lambda_{n}}\right)\right\} ,\end{align*}
 and\begin{align*}
A_{q}(q^{-4n}z) & =\exp\left\{ \pi\lambda_{n}\left(v+\frac{2n}{\lambda_{n}}\right)^{2}-\frac{\pi\lambda_{n}}{12}-\frac{\pi}{24\lambda_{n}}\right\} \\
 & \times\sqrt{2}\left\{ \cos\pi\lambda_{n}v+\mathcal{O}\left(e^{-2\pi\lambda_{n}}\right)\right\} \end{align*}
as $n\to\infty$, and the $\mathcal{O}$-term is uniform for $v$
in any compact subset of $\mathbb{R}$. 
\end{cor}
For the Jackson's $q$-Bessel function we have:
\begin{cor}
\label{cor:jackson}For an admissible scale $\lambda_{n}$, assume
that\begin{align*}
z & =e^{2\pi v},\quad q=e^{-\pi\lambda_{n}^{-1}},\quad v\in\mathbb{R},\quad\nu>-1,\end{align*}
 then, 

\begin{align*}
J_{\nu}^{(2)}(2i\sqrt{zq^{-\nu}}q^{-2n};q) & =\frac{\exp\left(\frac{\pi\lambda_{n}}{3}-\frac{\pi}{12\lambda_{n}}+\frac{\nu^{2}\pi}{4\lambda_{n}}+\frac{\nu\pi i}{2}\right)}{2\sqrt{\lambda_{n}}}\\
 & \times\exp\left\{ \pi\lambda_{n}\left(v+\frac{4n+\nu}{2\lambda_{n}}\right)^{2}\right\} \\
 & \times\left\{ 1+\mathcal{O}(e^{-\pi\lambda_{n}})\right\} ,\end{align*}

and\begin{align*}
J_{\nu}^{(2)}(2\sqrt{zq^{-\nu}}q^{-2n};q) & =\frac{\exp\left(\frac{\pi\lambda_{n}}{12}-\frac{\pi}{12\lambda_{n}}+\frac{\nu^{2}\pi}{4\lambda_{n}}\right)}{\sqrt{\lambda_{n}}}\\
 & \times\exp\left\{ \pi\lambda_{n}\left(v+\frac{4n+\nu}{2\lambda_{n}}\right)^{2}\right\} \\
 & \times\left\{ \cos\pi\lambda_{n}v+\mathcal{O}\left(e^{-2\pi\lambda_{n}}\right)\right\} \end{align*}
 as $n\to\infty$, and the $\mathcal{O}$-term is uniform for $v$
in any compact subset of $\mathbb{R}$. 
\end{cor}
For the confluent basic hypergeometric series we have:
\begin{cor}
\label{cor:confluent}Given an admissible scale $\lambda_{n}$, assume
that

\begin{align*}
z & =e^{2\pi v},\quad q=e^{-\pi\lambda_{n}^{-1}},\quad v\in\mathbb{R},\end{align*}
and\[
\alpha_{j},\beta_{k}>0,\quad1\le j\le r,\quad1\le k\le s.\]
Let \begin{align*}
\ell & =\frac{s+1-r}{2}>0,\quad\rho=\sum_{j=1}^{r}\alpha_{j}-\sum_{k=1}^{s}\beta_{k}-1,\end{align*}
 then,\begin{align*}
 & {}_{r}\phi_{s}\left(\begin{array}{c}
q^{\alpha_{1}},\dotsc,q^{\alpha_{r}}\\
q^{\beta_{1}},\dotsc,q^{\beta_{s}}\end{array}\vert q,(-1)^{s+1-r}zq^{-\ell(4n-1)}\right)\\
 & =\frac{\prod_{k=1}^{s}\Gamma(\beta_{k})}{\prod_{j=1}^{r}\Gamma(\alpha_{j})}\frac{\lambda_{n}^{\rho+\ell+1/2}}{\sqrt{\ell}2^{\ell}\pi^{\rho+2\ell}}\\
 & \times\left\{ \exp\frac{\pi\lambda_{n}}{\ell}\left(v+\frac{2n\ell}{\lambda_{n}}\right)^{2}+\ell\pi\lambda_{n}/3\right\} \\
 & \times\left\{ 1+\mathcal{O}\left(\lambda_{n}^{-1}\right)\right\} ,\end{align*}
 and\begin{align*}
 & {}_{r}\phi_{s}\left(\begin{array}{c}
q^{\alpha_{1}},\dotsc,q^{\alpha_{r}}\\
q^{\beta_{1}},\dotsc,q^{\beta_{s}}\end{array}\vert q,(-1)^{s-r}zq^{-\ell(4n-1)}\right)\\
 & =\frac{\prod_{k=1}^{s}\Gamma(\beta_{k})}{\prod_{j=1}^{r}\Gamma(\alpha_{j})}\frac{\lambda_{n}^{\rho+\ell+1/2}}{\sqrt{\ell}2^{\ell-1}\pi^{\rho+2\ell}}\\
 & \times\left\{ \exp\frac{\pi\lambda_{n}}{\ell}\left(v+\frac{2n\ell}{\lambda_{n}}\right)^{2}+\frac{\ell\pi\lambda_{n}}{3}-\frac{\pi\lambda_{n}}{4\ell}\right\} \\
 & \times\left\{ \cos\frac{\pi\lambda_{n}v}{\ell}+\mathcal{O}\left(\lambda_{n}^{-1}\right)\right\} \end{align*}
 as $n\to\infty$, and the $\mathcal{O}$-term is uniform for $v$
in any compact subset of $\mathbb{R}$. 
\end{cor}

\subsection{$h$-function}

For our convenience we let\begin{align*}
h_{n}(z;q) & =h_{\ell}(a_{1},\dotsc,a_{r};b_{1},\dotsc,b_{s};c_{1},\dots,c_{t};q;z).\end{align*}
We have similar results for the $h$ function:
\begin{thm}
\label{thm:h-asymptotics-1}Given an admissible scale $\lambda_{n}$
, assume that \begin{align*}
z & =e^{2\pi v},\quad q=e^{-\pi\lambda_{n}^{-1}},\quad\ell>0,\quad v\in\mathbb{R},\end{align*}
 and\begin{align*}
a_{j} & =q^{\alpha_{j}},\quad b_{k}=q^{\beta_{k}},\quad\alpha_{j},\beta_{k}>0\end{align*}
for \begin{align*}
1 & \le j\le r,\quad1\le k\le s.\end{align*}
Then, \begin{align*}
h(-zq^{-n\ell};q) & =\exp\left\{ \frac{\pi\lambda_{n}}{\ell}\left(v+\frac{\ell(n-\chi(n))}{2\lambda_{n}}\right)^{2}+\frac{\ell\pi(n-1)\chi(n)}{2\lambda_{n}}\right\} \\
 & \times\sqrt{\frac{\lambda_{n}}{\ell}}\left\{ 1+\mathcal{O}(e^{-\ell^{-1}\pi\lambda_{n}})\right\} ,\end{align*}
 and\begin{align*}
h(zq^{-n\ell};q) & =\exp\left\{ \frac{\pi\lambda_{n}}{\ell}\left(v+\frac{\ell(n-\chi(n))}{2\lambda_{n}}\right)^{2}+\frac{\ell\pi(n-1)\chi(n)}{2\lambda_{n}}-\frac{\pi\lambda_{n}}{4\ell}\right\} \\
 & \times2\sqrt{\frac{\lambda_{n}}{\ell}}\left\{ \cos\frac{\pi\lambda_{n}}{\ell}\left(v+\frac{\ell(n-\chi(n))}{2\lambda_{n}}\right)+\mathcal{O}(e^{-2\ell^{-1}\pi\lambda_{n}})\right\} \end{align*}
as $n\to\infty$, and the $\mathcal{O}$-term is uniform for $v$
in any compact subset of $\mathbb{R}$. 
\end{thm}
For Ismail-Masson orthogonal polynomials we have:
\begin{cor}
\label{cor:ismail-masson} Given an admissible scale $\lambda_{n}$,
for any $v\in\mathbb{R}$, we have\begin{align*}
h_{n}\left(\sinh\pi\left(v+\frac{i}{2}\right)\mid q\right) & =\frac{\exp\left\{ \frac{\pi n^{2}}{4\lambda_{n}}+\frac{\pi\lambda_{n}}{6}-\frac{\pi(1+12\chi(n))}{24\lambda_{n}}\right\} }{(-i)^{n}\sqrt{2}}\\
 & \times\left\{ \exp\left[\pi\lambda_{n}\left(v-\frac{\chi(n)}{2\lambda_{n}}\right)^{2}\right]\right\} \left\{ 1+\mathcal{O}(e^{-\pi\lambda_{n}})\right\} ,\end{align*}
 and\begin{align*}
h_{n}(\sinh\pi v & \mid q)=(-1)^{n}\sqrt{2}\exp\left\{ \frac{n^{2}\pi}{4\lambda_{n}}-\frac{(1+12\chi(n))\pi}{24\lambda_{n}}-\frac{\pi\lambda_{n}}{12}\right\} \\
 & \times\left\{ \exp\left[\pi\lambda_{n}\left(v-\frac{\chi(n)}{2\lambda_{n}}\right)^{2}\right]\right\} \\
 & \times\left\{ \cos\pi\lambda_{n}\left(v+\frac{n-\chi(n)}{2\lambda_{n}}\right)+\mathcal{O}(e^{-2\pi\lambda_{n}})\right\} \end{align*}

as $n\to\infty$, and the $\mathcal{O}$-term is uniform for $v$
in any compact subset of $\mathbb{R}$. 
\end{cor}
For Stieltjes-Wigert orthogonal polynomials we have:
\begin{cor}
\label{cor:stieltjes-wigert}Given an admissible scale $\lambda_{n}$,
assume that \begin{align*}
z & =e^{2\pi v},\quad q=e^{-\pi\lambda_{n}^{-1}},\quad v\in\mathbb{R}.\end{align*}
Then,\begin{align*}
S_{n}(-zq^{-n};q) & =\frac{\exp\left\{ \frac{\pi\lambda_{n}}{3}+\frac{\pi(n-1)\chi(n)}{2\lambda_{n}}-\frac{\pi}{12\lambda_{n}}\right\} }{2\sqrt{\lambda_{n}}}\\
 & \times\left\{ \exp\pi\lambda_{n}\left(v+\frac{n-\chi(n)}{2\lambda_{n}}\right)^{2}\right\} \left\{ 1+\mathcal{O}\left(e^{-\pi\lambda_{n}}\right)\right\} ,\end{align*}
and\begin{align*}
S_{n}(zq^{-n};q) & =\frac{\exp\left\{ \frac{\pi\lambda_{n}}{12}+\frac{\pi(n-1)\chi(n)}{2\lambda_{n}}-\frac{\pi}{12\lambda_{n}}\right\} }{\sqrt{\lambda_{n}}}\\
 & \times\left\{ \exp\pi\lambda_{n}\left(v+\frac{n-\chi(n)}{2\lambda_{n}}\right)^{2}\right\} \\
 & \times\left\{ \cos\pi\lambda_{n}\left(v+\frac{n-\chi(n)}{2\lambda_{n}}\right)+\mathcal{O}\left(e^{-2\pi\lambda_{n}}\right)\right\} \end{align*}
as $n\to\infty$, and the $\mathcal{O}$-term is uniform for $v$
in any compact subset of $\mathbb{R}$. 
\end{cor}
For the $q$-Laguerre orthogonal polynomials we have:
\begin{cor}
\label{cor:laguerre}Given an admissible scale $\lambda_{n}$, assume
that\begin{align*}
z & =e^{-2\pi v},\quad q=e^{-\pi\lambda_{n}^{-1}},\quad v\in\mathbb{R},\quad\alpha>-1.\end{align*}
 Then,\begin{align*}
L_{n}^{(\alpha)}(-zq^{-\alpha-n};q) & =\frac{\exp\left\{ \frac{\pi\lambda_{n}}{3}+\frac{\pi(n-1)\chi(n)}{2\lambda_{n}}-\frac{\pi}{12\lambda_{n}}\right\} }{2\sqrt{\lambda_{n}}}\\
 & \times\left\{ \exp\pi\lambda_{n}\left(v+\frac{n-\chi(n)}{2\lambda_{n}}\right)^{2}\right\} \left\{ 1+\mathcal{O}\left(e^{-\pi\lambda_{n}}\right)\right\} ,\end{align*}
 and\begin{align*}
L_{n}^{(\alpha)}(zq^{-\alpha-n};q) & =\frac{\exp\left\{ \frac{\pi\lambda_{n}}{12}+\frac{\pi(n-1)\chi(n)}{2\lambda_{n}}-\frac{\pi}{12\lambda_{n}}\right\} }{\sqrt{\lambda_{n}}}\\
 & \times\left\{ \exp\pi\lambda_{n}\left(v+\frac{n-\chi(n)}{2\lambda_{n}}\right)^{2}\right\} \\
 & \times\left\{ \cos\pi\lambda_{n}\left(v+\frac{n-\chi(n)}{2\lambda_{n}}\right)+\mathcal{O}\left(e^{-2\pi\lambda_{n}}\right)\right\} \end{align*}
as $n\to\infty$, and the $\mathcal{O}$-term is uniform for $v$
in any compact subset of $\mathbb{R}$. \end{cor}
\begin{rem}
Similar results hold for general $\tau$ and $\beta$ defined in \cite{Zhang1}
and their proofs are also similar to the proofs for the special cases
here. However, we feel that the formulas for the special cases are
more appealing and thus skip the general formulas. 
\end{rem}

\section{Proofs\label{sec:Proofs}}

The following lemma is from \cite{Zhang1}:
\begin{lem}
\label{lem:1}Given $a\in\mathbb{C}$ with\begin{align*}
0 & <\frac{\left|a\right|q^{n}}{1-q}<\frac{1}{2}\end{align*}
 for some $n\in\mathbb{N}$. Then, \begin{align*}
\frac{(a;q)_{\infty}}{(a;q)_{n}} & =(aq^{n};q)_{\infty}=1+r_{1}(a;n)\end{align*}
 with\begin{align*}
\left|r_{1}(a;n)\right| & \le\frac{2\left|a\right|q^{n}}{1-q}\end{align*}
 and \begin{align*}
\frac{(a;q)_{n}}{(a;q)_{\infty}} & =\frac{1}{(aq^{n};q)_{\infty}}=1+r_{2}(a;n)\end{align*}
 with \begin{align*}
\left|r_{2}(a;n)\right| & \le\frac{2\left|a\right|q^{n}}{1-q}.\end{align*}

\end{lem}
We also need the following lemma:
\begin{lem}
\label{lem:3}Given a sequence of positive numbers $\left\{ \lambda_{n}\right\} _{n=1}^{\infty}$
with $\lim_{n\to\infty}\lambda_{n}=\infty$. let \begin{align*}
q & =e^{-\pi\lambda_{n}^{-1}},\quad x>0,\end{align*}
 then,\begin{align*}
(q^{x};q)_{\infty} & =\frac{\sqrt{2}\pi^{1-x}\lambda_{n}^{x-1/2}}{\Gamma(x)\exp(\pi\lambda_{n}/6)}\left\{ 1+\mathcal{O}\left(\lambda_{n}^{-1}\right)\right\} \end{align*}
 as $n\to\infty$.\end{lem}
\begin{proof}
For $x\neq0,-1,-2,\dots$ and let $q=e^{-t}$, the McIntosh asymptotic
formula \cite{McIntosh} is\begin{align*}
\log(q^{x};q)_{\infty} & =-\frac{\pi^{2}}{6t}+\left(\frac{1}{2}-x\right)\log t+\frac{\log(2\pi)}{2}-\log\Gamma(x)\\
 & +\sum_{k=1}^{p}\frac{B_{k}B_{k+1}(x)}{k(k+1)!}t^{k}+\mathcal{O}\left(t^{p+1}\right)\end{align*}
for any positive integer $p$ as $t\to0^{+}$ , where $B_{k}$ is
the $k^{th}$ Bernoulli number and $B_{k}(x)$ is the $k^{th}$ Bernoulli
polynomial. Take the main term in the McIntosh asymptotic formula
with $t=\frac{\pi}{\lambda_{n}}$ and Lemma \ref{lem:3} follows.
\end{proof}
Take $\lambda=0,\quad\tau=2,\quad m=2n$ in Theorem 2.2 of \cite{Zhang2}
we get the following result:
\begin{lem}
\label{lem:4}Assume that $z\in\mathbb{C}\backslash\left\{ 0\right\} $,
$\ell>0$ and \eqref{eq:g-h condition}, then, \begin{align*}
g(q^{-4n\ell}z;q) & =z^{2n}q^{-4n^{2}\ell}\left\{ \theta_{4}\left(z^{-1};q^{\ell}\right)+r_{g}(n|1)\right\} ,\end{align*}
 and\begin{align*}
|r_{g}(n|1)| & \le\frac{2^{s+r+3}\theta_{3}\left(|z|^{-1};q^{\ell}\right)}{(a_{1},\dotsc,a_{r};q)_{\infty}}\left\{ \frac{q^{n+1}}{1-q}+\frac{q^{\ell n^{2}}}{\left|z\right|^{n}}\right\} \end{align*}
 for $n$ sufficiently large. In particular, \begin{align*}
 & {}_{r}\phi_{s}\left(\begin{array}{c}
a_{1},\dotsc,a_{r}\\
b_{1},\dotsc,b_{s}\end{array}\vert q,(-1)^{s-r}zq^{-4n\ell}\right)\\
 & =\frac{(a_{1},\dotsc,a_{r};q)_{\infty}z^{2n}\left\{ \theta_{4}\left(z^{-1}q^{\ell};q^{\ell}\right)+r_{\phi}(n|1)\right\} }{(q,b_{1},\dotsc,b_{s};q)_{\infty}q^{2\ell n(2n+1)}},\end{align*}
 and\begin{align*}
|r_{\phi}(n|1)| & \le\frac{2^{s+r+3}\theta_{3}\left(|z|^{-1}q^{\ell};q^{\ell}\right)}{(a_{1},\dotsc,a_{r};q)_{\infty}}\left\{ \frac{q^{n+1}}{1-q}+\frac{q^{\ell n^{2}+\ell n}}{\left|z\right|^{n}}\right\} \end{align*}
 for $n$ sufficiently large, where \begin{align*}
\ell & =\frac{s+1-r}{2}>0.\end{align*}

\end{lem}
Similarly, if we take $\lambda=0$ and $\tau=\frac{1}{2}$ in Theorem
2.4 of \cite{Zhang2} we get 
\begin{lem}
\label{lem:5}Assume that $z\in\mathbb{C}\backslash\left\{ 0\right\} $,
$\ell>0$ and \eqref{eq:g-h condition}, then \begin{align*}
h_{n}(zq^{-n\ell};q) & =(-z)^{\left\lfloor n/2\right\rfloor }q^{-\ell\left[n^{2}-\chi(n)\right]/4}\left\{ \theta_{4}\left(z^{-1};q^{\ell}\right)+r_{h}(n|1)\right\} ,\end{align*}
 and\begin{align*}
|r_{h}(n|1)| & \le\frac{2^{s+r+2t+5}\theta_{3}(|z|^{-1};q^{\ell})}{(a_{1},\dotsc,a_{r};q)_{\infty}}\\
 & \times\left\{ \frac{q^{\left\lfloor n/4\right\rfloor +1}}{1-q}+|z|^{\left\lfloor n/4\right\rfloor }q^{\ell\left\lfloor n/4\right\rfloor ^{2}}+\frac{q^{\ell\left\lfloor n/4\right\rfloor ^{2}}}{|z|^{\left\lfloor n/4\right\rfloor }}\right\} \end{align*}
for $n$ sufficiently large.
\end{lem}

\subsection{Proof for Theorem \ref{thm:g-asymptotics-1}}

From the formulas of $\theta_{3}$ to obtain\begin{align}
\theta_{3}(e^{-2\pi v};e^{-\pi\ell\lambda_{n}^{-1}}) & =\theta_{3}\left(vi|\ell\lambda_{n}^{-1}i\right)\label{eq:104}\\
 & =\sqrt{\frac{\lambda_{n}}{\ell}}e^{\pi\ell^{-1}\lambda_{n}v^{2}}\theta_{3}\left(\frac{\lambda_{n}v}{\ell}\mid\frac{\lambda_{n}i}{\ell}\right)\nonumber \\
 & =\sqrt{\frac{\lambda_{n}}{\ell}}e^{\pi\ell^{-1}\lambda_{n}v^{2}}\left\{ 1+\mathcal{O}(e^{-\ell^{-1}\pi\lambda_{n}})\right\} \nonumber \end{align}
as $n\to\infty$, uniformly for all $v\in\mathbb{R}$. Clearly we
have\begin{align*}
\frac{q^{n+1}}{1-q}+q^{\ell n^{2}}e^{-2n\pi v} & =\mathcal{O}(\lambda_{n}e^{-\pi n\lambda_{n}^{-1}})\end{align*}
as $n\to\infty$, uniformly for $v$ in any compact subset of $\mathbb{R}$.
From Lemma \ref{lem:3} we have\begin{align}
(q^{\alpha_{1}},\dots,q^{\alpha_{r}};q)_{\infty} & =\frac{2^{r/2}\pi^{r-\sum_{j=1}^{r}\alpha_{j}}\left\{ 1+\mathcal{O}(\lambda_{n}^{-1})\right\} }{e^{r\pi\lambda_{n}/6}\lambda_{n}^{r/2-\sum_{j=1}^{r}\alpha_{j}}\prod_{j=1}^{r}\Gamma(\alpha_{j})}\label{eq:106}\end{align}
as $n\to\infty$. Condition \eqref{eq:admissible} gives\begin{align*}
g(-q^{-4n\ell}z;q) & =\sqrt{\frac{\lambda_{n}}{\ell}}\exp\left\{ \frac{\pi\lambda_{n}}{\ell}\left(v+\frac{2n\ell}{\lambda_{n}}\right)^{2}\right\} \left\{ 1+\mathcal{O}(e^{-\ell^{-1}\pi\lambda_{n}})\right\} \end{align*}
 as $n\to\infty$, uniformly for $v$ in any compact subset of $\mathbb{R}$.

Since \begin{align}
\theta_{4}\left(z^{-1};q^{\ell}\right) & =\theta_{4}\left(e^{-2\pi v};e^{-\ell\pi\lambda_{n}^{-1}}\right)\label{eq:108}\\
 & =\theta_{4}\left(vi\mid\frac{\ell i}{\lambda_{n}}\right)\nonumber \\
 & =\sqrt{\frac{\lambda_{n}}{\ell}}e^{\pi\ell^{-1}\lambda_{n}v^{2}}\theta_{2}\left(\frac{\lambda_{n}v}{\ell}\mid\frac{i\lambda_{n}}{\ell}\right)\nonumber \\
 & =2\sqrt{\frac{\lambda_{n}}{\ell}}\exp\left(\frac{\pi\lambda_{n}v^{2}}{\ell}-\frac{\pi\lambda_{n}}{4\ell}\right)\nonumber \\
 & \times\cos\frac{\pi\lambda_{n}v}{\ell}\left\{ 1+\mathcal{O}\left(e^{-2\pi\ell^{-1}\lambda_{n}}\right)\right\} ,\nonumber \end{align}
 as $n\to\infty$, uniformly in $v\in\mathbb{R}$. Thus,\begin{align*}
g(q^{-4n\ell}z;q) & =\exp\left\{ \frac{\pi\lambda_{n}}{\ell}\left(v+\frac{2n\ell}{\lambda_{n}}\right)^{2}-\frac{\pi\lambda_{n}}{4\ell}\right\} \\
 & \times2\sqrt{\frac{\lambda_{n}}{\ell}}\left\{ \cos\frac{\pi\lambda_{n}v}{\ell}+\mathcal{O}\left(e^{-2\ell^{-1}\pi\lambda_{n}}\right)\right\} \end{align*}
 as $n\to\infty$, uniformly on any compact subset of $\mathbb{R}$.

\subsection{Proof for Corollary \ref{cor:ramanujan}}

By Lemma \ref{lem:3} and Theorem \ref{thm:g-asymptotics-1} we have\begin{align*}
A_{q}(-q^{-4n}z) & =\frac{g(-;-;q;1;-q^{-4n}z)}{(q;q)_{\infty}}\\
 & =\exp\left\{ \pi\lambda_{n}\left(v+\frac{2n}{\lambda_{n}}\right)^{2}+\frac{\pi\lambda_{n}}{6}-\frac{\pi}{24\lambda_{n}}\right\} \\
 & \times\frac{1}{\sqrt{2}}\left\{ 1+\mathcal{O}\left(e^{-\pi\lambda_{n}}\right)\right\} ,\end{align*}
and\begin{eqnarray*}
A_{q}(q^{-4n}z) & = & \frac{g(-;-;q;1;q^{-4n}z)}{(q;q)_{\infty}}\\
 & = & \exp\left\{ \pi\lambda_{n}\left(v+\frac{2n}{\lambda_{n}}\right)^{2}-\frac{\pi\lambda_{n}}{12}-\frac{\pi}{24\lambda_{n}}\right\} \\
 & \times & \sqrt{2}\left\{ \cos\pi\lambda_{n}v+\mathcal{O}\left(e^{-2\pi\lambda_{n}}\right)\right\} \end{eqnarray*}
 as $n\to\infty$, uniformly on any compact subset of $\mathbb{R}$.

\subsection{Proof for Corollary \ref{cor:jackson}}

Apply Lemma \ref{lem:3} and Theorem \ref{thm:g-asymptotics-1} to
get

\begin{align*}
J_{\nu}^{(2)}(2i\sqrt{zq^{-\nu}}q^{-2n};q) & =\frac{g(-;q^{\nu+1};q;1;-zq^{-4n})}{(q;q)_{\infty}^{2}(i\sqrt{zq^{-\nu}}q^{-2n})^{-\nu}}\\
 & =\frac{\exp\left(\frac{\pi\lambda_{n}}{3}-\frac{\pi}{12\lambda_{n}}+\frac{\nu^{2}\pi}{4\lambda_{n}}+\frac{\nu\pi i}{2}\right)}{2\sqrt{\lambda_{n}}}\\
 & \times\exp\left\{ \pi\lambda_{n}\left(v+\frac{4n+\nu}{2\lambda_{n}}\right)^{2}\right\} \\
 & \times\left\{ 1+\mathcal{O}(e^{-\pi\lambda_{n}})\right\} ,\end{align*}
 and \begin{align*}
J_{\nu}^{(2)}(2\sqrt{zq^{-\nu}}q^{-2n};q) & =\frac{g(-;q^{\nu+1};q;1;zq^{-4n})}{(q;q)_{\infty}^{2}(\sqrt{zq^{-\nu}}q^{-2n})^{-\nu}}\\
 & =\frac{\exp\left(\frac{\pi\lambda_{n}}{12}-\frac{\pi}{12\lambda_{n}}+\frac{\nu^{2}\pi}{4\lambda_{n}}\right)}{\sqrt{\lambda_{n}}}\\
 & \times\exp\left\{ \pi\lambda_{n}\left(v+\frac{4n+\nu}{2\lambda_{n}}\right)^{2}\right\} \\
 & \times\left\{ \cos\pi\lambda_{n}v+\mathcal{O}\left(e^{-2\pi\lambda_{n}}\right)\right\} \end{align*}
 as $n\to\infty$, uniformly on any compact subset of $\mathbb{R}$.

\subsection{Proof for Corollary \ref{cor:confluent}}

Apply Lemma \ref{lem:3} , Lemma \ref{lem:3} and Theorem \ref{thm:g-asymptotics-1}
to get\begin{align*}
 & {}_{r}\phi_{s}\left(\begin{array}{c}
q^{\alpha_{1}},\dotsc,q^{\alpha_{r}}\\
q^{\beta_{1}},\dotsc,q^{\beta_{s}}\end{array}\vert q,(-1)^{s+1-r}zq^{-\ell(4n-1)}\right)\\
 & =\frac{(q^{\alpha_{1}},\dots,q^{\alpha_{r}};q)_{\infty}g(-q^{-4n\ell}z;q)}{(q,q^{\beta_{1}},\dots,q^{\beta_{s}};q)_{\infty}}\\
 & =\frac{\prod_{k=1}^{s}\Gamma(\beta_{k})}{\prod_{j=1}^{r}\Gamma(\alpha_{j})}\frac{\lambda_{n}^{\rho+\ell+1/2}}{2^{\ell}\pi^{\rho+2\ell}\sqrt{\ell}}\\
 & \times\left\{ \exp\frac{\pi\lambda_{n}}{\ell}\left(v+\frac{2n\ell}{\lambda_{n}}\right)^{2}+\ell\pi\lambda_{n}/3\right\} \\
 & \times\left\{ 1+\mathcal{O}\left(\lambda_{n}^{-1}\right)\right\} ,\end{align*}
 and\begin{align*}
 & {}_{s}\phi_{r}\left(\begin{array}{c}
q^{\alpha_{1}},\dotsc,q^{\alpha_{r}}\\
q^{\beta_{1}},\dotsc,q^{\beta_{s}}\end{array}\vert q,(-1)^{s-r}zq^{-\ell(4n-1)}\right)\\
 & =\frac{(q^{\alpha_{1}},\dots,q^{\alpha_{r}};q)_{\infty}g(q^{-4n\ell}z;q)}{(q,q^{\beta_{1}},\dots,q^{\beta_{s}};q)_{\infty}}\\
 & =\frac{\prod_{k=1}^{s}\Gamma(\beta_{k})}{\prod_{j=1}^{r}\Gamma(\alpha_{j})}\frac{\lambda_{n}^{\rho+\ell+1/2}}{\sqrt{\ell}2^{\ell-1}\pi^{\rho+2\ell}}\\
 & \times\left\{ \exp\frac{\pi\lambda_{n}}{\ell}\left(v+\frac{2n\ell}{\lambda_{n}}\right)^{2}+\frac{\ell\pi\lambda_{n}}{3}-\frac{\pi\lambda_{n}}{4\ell}\right\} \\
 & \times\left\{ \cos\frac{\pi\lambda_{n}v}{\ell}+\mathcal{O}\left(\lambda_{n}^{-1}\right)\right\} .\end{align*}

\subsection{Proof for Theorem \ref{thm:h-asymptotics-1}}

Clearly, \begin{align}
\frac{q^{\left\lfloor n/4\right\rfloor +1}}{1-q}+|z|^{\left\lfloor n/4\right\rfloor }q^{\ell\left\lfloor n/4\right\rfloor ^{2}}+\frac{q^{\ell\left\lfloor n/4\right\rfloor ^{2}}}{|z|^{\left\lfloor n/4\right\rfloor }} & =\mathcal{O}\left(e^{-\pi n/(4\lambda_{n})}\right)\label{eq:116}\end{align}
as $n\to\infty$, uniformly on any compact subset of $\mathbb{R}$.
From equations \eqref{eq:104}, \eqref{eq:106} and \eqref{eq:116}
to get\begin{align*}
h(-zq^{-n\ell};q) & =\exp\left\{ \frac{\pi\lambda_{n}}{\ell}\left(v+\frac{\ell(n-\chi(n))}{2\lambda_{n}}\right)^{2}+\frac{\ell\pi(n-1)\chi(n)}{2\lambda_{n}}\right\} \\
 & \times\sqrt{\frac{\lambda_{n}}{\ell}}\left\{ 1+\mathcal{O}(e^{-\ell^{-1}\pi\lambda_{n}})\right\} \end{align*}
 as $n\to\infty$, uniformly on any compact subset of $\mathbb{R}$.
Using equations \eqref{eq:104}, \eqref{eq:106}, \eqref{eq:108}
and \eqref{eq:116} to obtain\begin{align*}
h(zq^{-n\ell};q) & =\exp\left\{ \frac{\pi\lambda_{n}}{\ell}\left(v+\frac{\ell(n-\chi(n))}{2\lambda_{n}}\right)^{2}+\frac{\ell\pi(n-1)\chi(n)}{2\lambda_{n}}-\frac{\pi\lambda_{n}}{4\ell}\right\} \\
 & \times2\sqrt{\frac{\lambda_{n}}{\ell}}\left\{ \cos\frac{\pi\lambda_{n}}{\ell}\left(v+\frac{\ell(n-\chi(n))}{2\lambda_{n}}\right)+\mathcal{O}(e^{-2\ell^{-1}\pi\lambda_{n}})\right\} \end{align*}
 as $n\to\infty$, uniformly on any compact subset of $\mathbb{R}$.

\subsection{Proof for Corollary \ref{cor:ismail-masson}}

For $v\in\mathbb{R}$, Lemma \ref{lem:3} and Theorem \ref{thm:h-asymptotics-1}
implies \begin{align*}
h_{n}\left(\sinh\pi\left(v+\frac{i}{2}\right)\mid q\right) & =\frac{h(-;-;-;q;1;-e^{2\pi v}q^{-n})}{(-i)^{n}e^{n\pi v}(q;q)_{\infty}}\\
 & =\frac{\exp\left\{ \frac{\pi n^{2}}{4\lambda_{n}}+\frac{\pi\lambda_{n}}{6}-\frac{\pi(1+12\chi(n))}{24\lambda_{n}}\right\} }{(-i)^{n}\sqrt{2}}\\
 & \times\left\{ \exp\left[\pi\lambda_{n}\left(v-\frac{\chi(n)}{2\lambda_{n}}\right)^{2}\right]\right\} \left\{ 1+\mathcal{O}(e^{-\pi\lambda_{n}})\right\} ,\end{align*}
 and\begin{align*}
h_{n}(\sinh\pi v & \mid q)=\frac{h(-;-;-;q;1;e^{2\pi v}q^{-n})}{(-1)^{n}e^{n\pi v}(q;q)_{\infty}}\\
 & =(-1)^{n}\sqrt{2}\exp\left\{ \frac{n^{2}\pi}{4\lambda_{n}}-\frac{(1+12\chi(n))\pi}{24\lambda_{n}}-\frac{\pi\lambda_{n}}{12}\right\} \\
 & \times\left\{ \exp\left[\pi\lambda_{n}\left(v-\frac{\chi(n)}{2\lambda_{n}}\right)^{2}\right]\right\} \\
 & \times\left\{ \cos\pi\lambda_{n}\left(v+\frac{n-\chi(n)}{2\lambda_{n}}\right)+\mathcal{O}(e^{-2\pi\lambda_{n}})\right\} \end{align*}
as $n\to\infty$, uniformly on any compact subset of $\mathbb{R}$.

\subsection{Proof for Corollary \ref{cor:stieltjes-wigert}}

From Lemma\ref{lem:1} and Lemma\ref{lem:3} to obtain\begin{align}
\frac{1}{(q;q)_{n}(q;q)_{\infty}} & =\frac{1}{(q;q)_{\infty}^{2}}\frac{(q;q)_{\infty}}{(q;q)_{n}}\label{eq:121}\\
 & =\frac{\exp\left\{ \frac{\pi\lambda_{n}}{3}-\frac{\pi}{12\lambda_{n}}\right\} }{2\lambda_{n}}\left\{ 1+\mathcal{O}\left(e^{-4\pi\lambda_{n}}\right)\right\} \nonumber \end{align}
as $n\to\infty$. Hence, Theorem \ref{thm:h-asymptotics-1} implies 

\begin{align*}
S_{n}(-zq^{-n};q) & =\frac{h(-;-;-;q;1;-zq^{-n})}{(q;q)_{n}(q;q)_{\infty}}\\
 & =\frac{\exp\left\{ \frac{\pi\lambda_{n}}{3}+\frac{\pi(n-1)\chi(n)}{2\lambda_{n}}-\frac{\pi}{12\lambda_{n}}\right\} }{2\sqrt{\lambda_{n}}}\\
 & \times\left\{ \exp\pi\lambda_{n}\left(v+\frac{n-\chi(n)}{2\lambda_{n}}\right)^{2}\right\} \left\{ 1+\mathcal{O}\left(e^{-\pi\lambda_{n}}\right)\right\} ,\end{align*}
and\begin{align*}
S_{n}(zq^{-n};q) & =\frac{h(-;-;-;q;1;zq^{-n})}{(q;q)_{n}(q;q)_{\infty}}\\
 & =\frac{\exp\left\{ \frac{\pi\lambda_{n}}{12}+\frac{\pi(n-1)\chi(n)}{2\lambda_{n}}-\frac{\pi}{12\lambda_{n}}\right\} }{\sqrt{\lambda_{n}}}\\
 & \times\left\{ \exp\pi\lambda_{n}\left(v+\frac{n-\chi(n)}{2\lambda_{n}}\right)^{2}\right\} \\
 & \times\left\{ \cos\pi\lambda_{n}\left(v+\frac{n-\chi(n)}{2\lambda_{n}}\right)+\mathcal{O}\left(e^{-2\pi\lambda_{n}}\right)\right\} \end{align*}
as $n\to\infty$, uniformly on any compact subset of $\mathbb{R}$.

\subsection{Proof for Corollary \ref{cor:laguerre}}

From \eqref{eq:121} and Theorem \ref{thm:h-asymptotics-1} to obtain

\begin{align*}
L_{n}^{(\alpha)}(-zq^{-\alpha-n};q) & =\frac{h(-;-;q^{\alpha+1};q;1;-zq^{-n})}{(q;q)_{n}(q;q)_{\infty}}\\
 & =\frac{\exp\left\{ \frac{\pi\lambda_{n}}{3}+\frac{\pi(n-1)\chi(n)}{2\lambda_{n}}-\frac{\pi}{12\lambda_{n}}\right\} }{2\sqrt{\lambda_{n}}}\\
 & \times\left\{ \exp\pi\lambda_{n}\left(v+\frac{n-\chi(n)}{2\lambda_{n}}\right)^{2}\right\} \left\{ 1+\mathcal{O}\left(e^{-\pi\lambda_{n}}\right)\right\} ,\end{align*}
 and\begin{align*}
L_{n}^{(\alpha)}(zq^{-\alpha-n};q) & =\frac{h(-;-;q^{\alpha+1};q;1;zq^{-n})}{(q;q)_{n}(q;q)_{\infty}}\\
 & =\frac{\exp\left\{ \frac{\pi\lambda_{n}}{12}+\frac{\pi(n-1)\chi(n)}{2\lambda_{n}}-\frac{\pi}{12\lambda_{n}}\right\} }{\sqrt{\lambda_{n}}}\\
 & \times\left\{ \exp\pi\lambda_{n}\left(v+\frac{n-\chi(n)}{2\lambda_{n}}\right)^{2}\right\} \\
 & \times\left\{ \cos\pi\lambda_{n}\left(v+\frac{n-\chi(n)}{2\lambda_{n}}\right)+\mathcal{O}\left(e^{-2\pi\lambda_{n}}\right)\right\} \end{align*}
as $n\to\infty$, uniformly on any compact subset of $\mathbb{R}$.

\thanks{The author is very grateful for the referee's comments, especially
for the reference \cite{McIntosh}.}

This work is partially supported by Chinese National Natural Science
Foundation grant No.10761002.

\end{document}